\documentclass[10pt]{article}
\font\smallit=cmti10

\usepackage{amssymb,latexsym,amsmath,epsfig,amsthm, mathtools, upgreek, listings} 
\usepackage{xcolor}

\newcommand{\absn}[1]{\left\| #1 \right\|}

\makeatletter

\renewcommand\section{\@startsection {section}{1}{\z@}
{-30pt \@plus -1ex \@minus -.2ex}
{2.3ex \@plus.2ex}
{\normalfont\normalsize\bfseries\boldmath}}

\renewcommand\subsection{\@startsection{subsection}{2}{\z@}
{-3.25ex\@plus -1ex \@minus -.2ex}
{1.5ex \@plus .2ex}
{\normalfont\normalsize\bfseries\boldmath}}

\renewcommand{\@seccntformat}[1]{\csname the#1\endcsname. }

\makeatother

\newtheorem{theorem}{Theorem}
\newtheorem{lemma}{Lemma}

\theoremstyle{definition}
\newtheorem{definition}{Definition}

\DeclarePairedDelimiter\autoset{\{}{\}}
\newcommand{\set}[1]{\autoset*{#1}}

\DeclarePairedDelimiter\absolute{|}{|}
\newcommand{\abs}[1]{\absolute*{#1}}

\DeclarePairedDelimiter\autobracket{(}{)}
\newcommand{\br}[1]{\autobracket*{#1}}

\begin{document}

\begin{center}
\uppercase{\bf On a family of cubic Thue Equations involving Fibonacci and Lucas numbers}
\vskip 20pt
{\bf Tobias Hilgart, Ingrid Vukusic and Volker Ziegler}\\
{\smallit Department of Mathematics, University of Salzburg, Austria}\\
\end{center}
\vskip 20pt
\vskip 30pt

\centerline{\bf Abstract}
\noindent
Let $F_n$ denote the $n$-th Fibonacci number and $L_n$ the $n$-th Lucas number. We completely solve the family of cubic Thue equations 
$${(X-F_nY)(X-L_nY)X-Y^3=\pm1}$$
and show that there are no non-trivial solutions for $n\neq 1,3$.

\pagestyle{myheadings}
\thispagestyle{empty}
\baselineskip=12.875pt
\vskip 30pt

\section{Introduction}

There is a vast amount of research regarding the resolution of Diophantine equations of the form
\[
    F(X,Y) = m,
\]
where $F \in \mathbb{Z}[X,Y]$ is an irreducible form of degree at least $3$ and $m \neq 0$ a fixed integer. A. Thue \cite{thue_init} proved that there are only finitely many solutions, but used non-effective methods and did not give bounds for the size of possible solutions. Based on his theory of linear forms in logarithms, A. Baker \cite{Baker-Thue} was able to give such effective bounds, which have been refined many times since.


E. Thomas \cite{thomas_init} studied parametrised families of Thue equations of the form
\[
    X\prod_{i=2}^N \br{ X-p_i(a)Y } - Y^N = \pm 1,
\]
with monic polynomials $p_i \in \mathbb{Z}[a]$. He conjectured that if $0 < \deg p_2 < \deg p_3 < \cdots < \deg p_N$, then there is a constant $a_0$ such that for all integers $a \geq a_0$ the equation has only solutions with $\abs{y} \leq 1$, which are called the trivial solutions.
Thomas proved his conjecture in the case $N=3$ under some technical hypothesis. Heuberger \cite{heub_thom} proved the conjecture in the general case, again under some technical hypothesis.
Indeed, at least some of the hypotheses are necessary, as Ziegler \cite{ziegler_ex} provided a counter example to Thomas' conjecture in the cubic case. 

In this paper, we consider a specific family of Thue equations. While Thomas and Heuberger considered polynomially parametrised families of Thue equations, we consider Thue equations which are parametrised by linear recurrence sequences. More specifically we consider the Fibonacci sequence defined by $F_0=0, F_1=1$ and $F_{n+2}=F_{n+1}+F_n$ for $n\geq 0$ and the Lucas sequence defined by $L_0=2, L_1=1$ and $L_{n+2}=L_{n+1}+L_n$ for $n\geq 0$ and we prove the following theorem:

\begin{theorem}\label{thm:main}
    Let $F_n$ denote the $n$-th Fibonacci number and $L_n$ the $n$-th Lucas number. The family of cubic Thue equations
    \begin{equation}\label{eq:thue}
        (X-F_nY)(X-L_nY)X-Y^3 = \pm 1
    \end{equation}
    has only the trivial solutions
    \begin{equation*}
        \pm\set{(1,0), (0,1), (F_n,1), (L_n,1)}
    \end{equation*}
    with the exception of $n \in \set{1,3}$. In those cases, one has the additional solutions $\pm\set{(2,1),(7,4)}$ and $\pm\set{(7,4), (38,273)}$ respectively.
\end{theorem}

For a fixed integer $n$ we will call a solution $(x,y)\in \mathbb{Z}^2$ to \eqref{eq:thue} \textit{trivial}, if $|y|\leq 1$.

The remainder of this paper is structured as follows. First, we state preliminary results which we use during our proof, including standard results regarding linear forms in logarithms as per Baker-Wüstholz \cite{baker_lb} or Bugeaud-Győry \cite{bugy_up} in Section 2. We then study properties of Equation \eqref{eq:thue}. Assuming the existence of non-trivial solutions $(x,y)$, we construct linear forms in logarithms in Section 3. Using Baker's method, we get a lower bound on $\log\abs{y}$ that is contradictory to Bugeaud's and Győry's upper bound. This will lead to a huge effective bound for $n$, which we then reduce using the LLL-algorithm in Section 4.

\section{Preliminary Results and Notation}
We start with Binet's formula for the $n$-th Fibonacci and Lucas number,
\[
    F_n = \frac{\alpha^n-\beta^n}{\sqrt{5}}, \quad L_n = \alpha^n+\beta^n,
\]
where
\[
    \alpha = \frac{1+\sqrt{5}}{2}, \quad \beta = \frac{1-\sqrt{5}}{2}.
\]
Note that
\[
    \beta = -\alpha^{-1}.
\]
Furthermore, the inequality
\begin{equation}\label{eq:fib-bnd}
    \alpha^{n-2} \leq F_n \pm 11\alpha^{-n} \leq \alpha^{n-1} 
\end{equation}
holds for all $n \geq 6$ and we also use the relations
\begin{equation}\label{eq:luc-fib}
    L_n = F_{n-1}+F_{n+1} = F_n + 2F_{n-1},
\end{equation}
a number of times throughout the paper.

In order to associate the roots of the polynomial $F(X,1)$ with the respective Fibonacci and Lucas numbers, we will use the following result due to Kriegl, Losik and Michor \cite[Theorem $5.1$]{root_choose}. We state a reformulated version due to Heuberger \cite[Theorem $14$, Corollary $15$]{heub_param_gen}.
\begin{theorem}\label{thm:root-analytic}
    Let $U$ be a neighbourhood of $0$ in $\mathbb{R}$, $a_i(t)$ be real analytic functions on $U$ for $1\leq i \leq N$ and
    \begin{equation*}
        P(t)(x) = x^N - a_1(t)x^{N-1} + \cdots + (-1)^N a_N(t), \;\; t \in U
    \end{equation*}
    such that for all $t \in U$ all roots of $P(t)(x) = 0$ are real.
    Then there are real analytic functions $x_i(t)$, $i = 1, \dots, N$ on $U$ such that these functions are roots of $P(t)$.
\end{theorem}

We will also use a very well known and simple result a number of times throughout the paper:
\begin{lemma}\label{lem:log-ub}
    Assume $x \in \mathbb{C}$ such that $\abs{x-1} \leq 1/2$, then $\log\abs{x} \leq 2 \abs{x-1}$.
\end{lemma}
\begin{proof}
Use the Taylor series expansion of $\log x$ at $1$.
\end{proof}

\begin{definition}
Let $f,g$ be two complex-valued functions. We write $f = L(g)$, if $\abs{f(x)} \leq \abs{g(x)}$ for all $x \in \mathbb{C}$.
\end{definition}

\begin{lemma}\label{lem:log-trick}
If $f(x) = 1+g(x)$ with $\abs{g(x)} \leq 1/2$, then $\log \abs{f} = L(2 g) $.
\end{lemma}
\begin{proof}
Follows immediately from Lemma $\ref{lem:log-ub}$.
\end{proof}

Before we state lower bounds on linear forms in logarithms, we first recall the definition of the logarithmic (Weil) height.

\begin{definition}
Let $\gamma$ be an algebraic number over $\mathbb{Q}$ with minimal polynomial
\[
    a_d X^d + \cdots + a_1 X + a_0 = a_d \prod_{i=1}^d \br{X-\gamma_i},
\]
with relatively prime integers $a_i$ and conjugates $\gamma_1 = \gamma, \gamma_2, \dots, \gamma_d$. The (absolute) logarithmic height of $\gamma$ is defined by
\[
    h(\gamma) := \frac{1}{d} \br{ \log\abs{a_d} + \sum_{i=1}^d \log\max\set{1,\abs{\gamma_i}} }.
\]
\end{definition}
\begin{theorem}[Baker, Wüstholz \cite{baker_lb}]\label{thm:baker-lb}
    Let $\gamma_1, \dots, \gamma_t$ be algebraic numbers not $0$ or $1$ in $K= \mathbb{Q}(\gamma_1,\dots,\gamma_t)$ of degree $D$, let $b_1, \dots, b_t \in \mathbb{Z}$ and let
    \[
        \Uplambda = b_1 \log\gamma_1 + \cdots + b_t \log\gamma_t
    \]
    be non-zero. Then
    \[
        \log\abs{\Uplambda} \geq - 18(t+1)!t^{t+1}(32D)^{t+2}\log(2tD) h_1 \cdots h_t \log B
    \]
    where 
    \[
        B \geq \max\set{\abs{b_1},\dots, \abs{b_t}}
    \]
    and
    \[
        h_i \geq \max\set{ h(\gamma_i), \abs{\log\gamma_i}D^{-1}, 0.16D^{-1} } \text{ for } 1 \leq i \leq t.
    \]
\end{theorem}

Finally, we also state the following result due to Bugeaud and Győry \cite{bugy_up}.
\begin{theorem}[Bugeaud, Gy\H{o}ry \cite{bugy_up}]\label{thm:bugy-ub}
    Let $B \geq \max\set{\abs{m}, e}$, $\alpha$ be a root of $F(X,1)$, $K:=\mathbb{Q}(\alpha)$, $R:=R_K$ the regulator of $K$ and $r$ the unit rank of $K$. Let $H\geq 3$ be an upper bound for the absolute values of the coefficients of $F$ and $N$ its degree.
    
    Then all solutions $(x,y) \in \mathbb{Z}^2$ of the Thue equation $F(X,Y) = m$ satisfy
    \[
        \max\set{\log\abs{x},\log\abs{y}} \leq 3^{r+27}(r+1)^{7r+19}N^{2N+6r+14} \cdot R \cdot \max\set{\log R, 1} \cdot \br{ R + \log(HB) }.
    \]
\end{theorem}

\section{Auxiliary Results}
We start by taking a look at the polynomial of Equation \eqref{eq:thue} at $Y=1$,
\begin{equation}\label{eq:defining-polynomial}
    f_n(X) := (X-F_n)(X-L_n)X - 1
\end{equation}

and its three roots $\alpha^{(1)}, \alpha^{(2)}, \alpha^{(3)}$.

\begin{lemma}\label{lem:root-aprox}
    Let $\alpha = (1+\sqrt{5})/2$, then the roots of $f_n$ are given by
    \begin{align*}
        \alpha^{(1)}
        &=\frac{1}{\sqrt{5}} \alpha^n - (-1)^n \frac{1}{\sqrt{5}} \alpha^{-n} + \frac{5}{1-\sqrt{5}} \alpha^{-2n} + L(12 \alpha^{-3n}) \\
         &=  F_n + L(6 \alpha^{-2n}), \\
        \alpha^{(2)} & = \alpha^n + (-1)^n \alpha^{-n} + L(4 \alpha^{-2n})= L_n + L(4 \alpha^{-2n}), \\
        \alpha^{(3)} &= \sqrt{5} \alpha^{-2n} + L(\alpha^{-4n}).
    \end{align*}
\end{lemma}
\begin{proof}
The form of $\alpha^{(i)}$ follows from Theorem \ref{thm:root-analytic}. As real analytic functions, they can each be expressed by a Laurent series
\[
    \alpha^{(i)} = \sum_{k = -N}^\infty d_k^{(i)} \alpha^{-kn}.
\]
Comparing the coefficients on both sides of the equation
\begin{equation}\label{eq:root-coef}
    (\alpha^{(i)}-F_n)(\alpha^{(i)}-L_n)\alpha^{(i)} = 1,
\end{equation}
which follows from $f_n\br{\alpha^{(i)}} = 0$, yields the coefficients $d_k^{(i)}$ to arbitrary depth.

Assume for a moment that $N > 1$. We only consider the highest order terms of the left hand side of Equation \eqref{eq:root-coef},
\[
    (d_{-N}^{(i)}\alpha^{Nn}+\cdots)(d_{-N}^{(i)}\alpha^{Nn}+\cdots)(d_{-N}^{(i)}\alpha^{Nn}+\cdots) = \br{ d_{-N}^{(i)} }^3 \alpha^{3Nn} + \cdots.
\]
Since the right hand side of Equation \eqref{eq:root-coef} only has a non-zero coefficient $1$ of $\alpha^0$, it immediately follows that $d_{-N}^{(i)} = 0$.

Thus we may assume $N = 1$. Equation \eqref{eq:root-coef} then becomes
\[
    \left(\br{d_{-1}^{(i)}-\frac{1}{\sqrt{5}}}\alpha^n+\cdots\right)\left(\br{d_{-1}^{(i)}-1}\alpha^n+\cdots\right)\br{ d_{-1}^{(i)} \alpha^n + \cdots } = 1
\]
and thus
\[
    \br{d_{-1}^{(i)}-\frac{1}{\sqrt{5}}}\br{d_{-1}^{(i)}-1}d_{-1}^{(i)} = 0,
\]
from which follows that the coefficient of $\alpha^n$ in $\alpha^{(i)}$ is $\frac{1}{\sqrt{5}}, 1$ or $0$. Similarly, we get the coefficient of $\alpha^0$ being $0$ and of $\alpha^{-n}$ being $(-1)^{n+1}\frac{1}{\sqrt{5}}, (-1)^n$ or $0$ respectively.

We calculate $d_2 = d_2^{(1)}$ for $\alpha^{(1)}$, already knowing that the first and second coefficients are the same as for $F_n$. Therefore we look at Equation \eqref{eq:root-coef} which is now
\[
    \br{ d_2 \alpha^{-2n} + \cdots }\left( \br{\frac{1}{\sqrt{5}}-1} \alpha^n + \cdots \right)\br{ \frac{1}{\sqrt{5}} \alpha^n + \cdots } = 1.
\]
This yields
\[
    d_2 \br{\frac{1}{\sqrt{5}}-1} \frac{1}{\sqrt{5}} = 1
\]
and thus
\[
    d_2 = \frac{5}{1-\sqrt{5}}.
\]
We have now calculated sufficiently many coefficients and hide the rest of the Laurent series inside the $L$-notation.
Note, for example, that $f_n(F_n) = -1 < 0$. If we either have $f_n(F_n + \kappa \alpha^{-2n}) > 0$ or $f_n(F_n - \kappa \alpha^{-2n}) > 0$, then the root $\alpha^{(1)}$ must lie in the interval $(F_n \pm \kappa \alpha^{-2n})$ due to the intermediate value theorem, from which $\alpha^{(1)} = F_n + L(\kappa \alpha^{-2n})$ follows. Indeed, we have  $f_n(F_n - 6 \alpha^{-2n}) > 0$ and thus $\alpha^{(1)} = F_n + L\br{6\alpha^{-2n}}$. The other assertions regarding the $L$-notation hold analogously. 
\end{proof}

As an immediate consequence, we also get the following lemma.
\begin{lemma}\label{lem:root-log-aprox}
    We have
    \begin{align*}
        \log\abs{\alpha^{(1)}} &= n \log\alpha - \log\sqrt{5} 
        +L\br{3\alpha^{-2n}}, \\
        \log\abs{\alpha^{(1)}-F_n} &= -2n \log\alpha + \log\br{\frac{5}{\sqrt{5}-1}} + L\br{ 6 \alpha^{-n} }, \\
        \log\abs{\alpha^{(2)}} &= n \log\alpha + L\br{ 4\alpha^{-2n} }, \\
        \log\abs{ \alpha^{(2)} - F_n } &= n\log\alpha + \log\br{ 1-\frac{1}{\sqrt{5}} } + L\br{ 6 \alpha^{-2n} }, \\
        \log\abs{\alpha^{(3)}} &= -2n \log\alpha + \log\sqrt{5} +
        L\br{\alpha^{-2n}}, \\
        \log\abs{ \alpha^{(3)} - F_n } &= n\log\alpha - \log\sqrt{5} + L\br{3\alpha^{-2n}}.
    \end{align*}
\end{lemma}
\begin{proof}
    Follows from Lemma \ref{lem:root-aprox} and Lemma \ref{lem:log-trick}. For example, we have
    \begin{align*}
        \log\abs{\alpha^{(1)} } &= \log\br{ F_n + L\br{6\alpha^{-2n}} } = \log\br{ \frac{1}{\sqrt{5}}\alpha^n \br{ 1 \pm \alpha^{-2n} + L\br{ 6\sqrt{5} \alpha^{-3n} } } } \\
        &= n\log\alpha - \log\sqrt{5} + \log\br{ 1 + L\br{ \frac{3}{2} \alpha^{-2n} } }
    \end{align*}
    and thus the claimed form for $\log\abs{\alpha^{(1)}}$ after using Lemma \ref{lem:log-trick}.
\end{proof}

\begin{lemma}
    The only solutions $(x,y) \in \mathbb{Z}^2$ to the Thue Equation \eqref{eq:thue} with $\abs{y} < 2$ are $\pm\set{ (1,0), (0,1), (F_n,1), (L_n,1) }$.
\end{lemma}
\begin{proof}
    If $y = 0$, then Equation \eqref{eq:thue} becomes $X^3 = \pm 1$ and thus $x = \pm 1$. If $y = \pm 1$, then either
    \[
        (X\pm F_n)(X\pm L_n)X = 0
    \]
    or
    \[
        (X\pm F_n)(X\pm L_n)X = \pm 2.
    \]
    The first case immediately gives the solutions $\pm \set{ (F_n,1), (L_n,1), (0,1) }$. The second case gives no further solutions, unless $L_n - F_n \leq 2$, which implies $n \leq 3$ but gives no new solutions.
\end{proof}

Throughout the remainder of this paper, let us assume that $(x,y)$ is a non-trivial solution to \eqref{eq:thue}, i.e. $\abs{y} \geq 2$. It is also advantageous to assume $n \geq 10$. This can be done without loss of generality, since we will solve the Thue Equation \eqref{eq:thue} for such small $n$ later.

For each of the $\alpha^{(i)}$ we define
\[
    \beta_i := x - \alpha^{(i)} y.
\]

Since $\beta_1\beta_2\beta_3 = \pm1$, they are all units. Furthermore, we have the following property for one of the $\beta_i$:

Let $j$ be defined by the $\beta_i$ with minimal absolute value, that is
\[
    \abs{\beta_j} = \min_{ i \in \set{1,2,3} } \abs{\beta_i},
\]
then we have the following result.

\begin{lemma}\label{lem:beta_j-ub}
    We have 
    \[
        \abs{\beta_j} \leq 2\alpha^5\abs{y}^{-2} \cdot \alpha^{-2n}.
    \]
\end{lemma}
\begin{proof}
    Let $i\neq j$, then
    \[
        \abs{y} \abs{\alpha^{(i)} - \alpha^{(j)}} \leq \abs{x-\alpha^{(i)}y} + \abs{x-\alpha^{(j)}y} \leq 2 \abs{x - \alpha^{(i)}y},  
    \]
    and thus
    \begin{equation}\label{eq:beta_i-lb1}
       \frac{\abs{y}}{2} \abs{\alpha^{(i)} - \alpha^{(j)}} \leq \abs{x - \alpha^{(i)}y}.
    \end{equation}
    
    For $\set{j,k,l} = \set{1,2,3}$, we have $\beta_j\beta_k\beta_l = 1$ and therefore
    \[
        \abs{\beta_j} = \abs{x-\alpha^{(k)}y}^{-1} \abs{x - \alpha^{(l)}y}^{-1}.
    \]
    Applying Inequality \eqref{eq:beta_i-lb1} on both factors on the right yields
    \begin{equation}\label{eq:beta_j-ub1}
        \abs{\beta_j} \leq 4\abs{y}^{-2} \abs{\alpha^{(k)}-\alpha^{(j)}}^{-1}\abs{\alpha^{(l)}-\alpha^{(j)}}^{-1}.
    \end{equation}
    Using case differentiation for the values of $j,k,l$ gives us
    \begin{align*}
        \abs{\alpha^{(k)} - \alpha^{(j)}}\abs{\alpha^{(l)} - \alpha^{(j)}} &\geq (L_n-F_n + L(10\alpha^{-2n}) )(F_n + L(9\alpha^{-2n})) \\
        &= 2 F_n F_{n-1} + L(11\alpha^{-n}).
    \end{align*}
    by Equation \eqref{eq:luc-fib}. We then use Inequality \eqref{eq:fib-bnd} to obtain the bound
    \[
        \abs{\alpha^{(k)} - \alpha^{(j)}}\abs{\alpha^{(l)} - \alpha^{(j)}} \geq 2 \alpha^{n-2}\alpha^{n-3} = 2\alpha^{-5} \alpha^{2n}.
    \]
    Combining this with Inequality \eqref{eq:beta_j-ub1} yields
    \[
        \abs{\beta_j} \leq 2\alpha^5\abs{y}^{-2} \cdot \alpha^{-2n}.
    \]
\end{proof}

From now on let $\set{j,k,l} = \set{1,2,3}$. We take a look at Siegel's identity, which in our case yields
\[
    \beta_j \br{ \alpha^{(k)} - \alpha^{(l)} } + \beta_l \br{ \alpha^{(j)} - \alpha^{(k)} } + \beta_k \br{ \alpha^{(l)} - \alpha^{(j)} } = 0.
\]
By rearranging terms, we arrive at
\begin{equation}\label{eq:siegl}
    \frac{\beta_l}{\beta_k} \frac{ \alpha^{(j)} - \alpha^{(k)} }{ \alpha^{(j)} - \alpha^{(l)} } - 1 = \frac{\beta_j}{\beta_k} \frac{ \alpha^{(l)} - \alpha^{(k)} }{ \alpha^{(j)} - \alpha^{(l)} } =: \gamma.
\end{equation}

\begin{lemma}\label{lem:gamma-ub}
    We have
    \[
        \abs{\gamma} \leq 8\alpha^8\abs{y}^{-3} \cdot \alpha^{-3n}.
    \]
\end{lemma}
\begin{proof}
    First, recall that we have $|\beta_j|\leq 2\alpha^5\abs{y}^{-2} \cdot \alpha^{-2n}$ by Lemma \ref{lem:beta_j-ub}. Next, note that for $i \neq j$, we have
    \begin{align}\label{eq:beta_k}
        \abs{\beta_i}
        &= \abs{ x-\alpha^{(i)}y }
        = \abs{ \br{x-\alpha^{(j)}y} + y\br{\alpha^{(j)}-\alpha^{(i)}} } \nonumber\\
        &= \abs{ \beta_j + y\br{\alpha^{(j)}-\alpha^{(i)}} } \nonumber\\
        &\geq \abs{y} \abs{\alpha^{(j)}-\alpha^{(i)} } - \abs{\beta_j}.
    \end{align}
    Moreover, we have
    \[
        \abs{\alpha^{(j)} - \alpha^{(i)}} \geq L_n-F_n + L\br{10\alpha^{-2n}} \geq 2F_{n-1} - 10\alpha^{-2n} \geq 2 \alpha^{n-3},
    \]
    by Inequality \eqref{eq:fib-bnd}. Thus, using Lemma \ref{lem:beta_j-ub}, we obtain from Inequality \eqref{eq:beta_k} that
    \begin{align*}\label{eq:beta_k-2}
        |\beta_i| 
        &\geq \abs{y} \cdot 2 \alpha^{n-3} - 2\alpha^5\cdot \abs{y}^{-2} \cdot \alpha^{-2n} \nonumber \\
        &= \abs{y}\alpha^{n-3}  \br{ 2 -   2\alpha^5 \abs{y}^{-3} \alpha^{-3n+3} }
        \geq \abs{y} \alpha^{n} \alpha^{-3} .
    \end{align*}
    We thus have
    \begin{equation}\label{eq:gamma-ub1}
        \abs{ \frac{\beta_j}{\beta_k} } \leq 2\alpha^5\abs{y}^{-2} \cdot \alpha^{-2n} \cdot \br{ \abs{y} \alpha^{n} \alpha^{-3} }^{-1} = 2\alpha^8\abs{y}^{-3} \cdot \alpha^{-3n}.
    \end{equation}
    
    For the second term of $\abs{\gamma}$, we differentiate the cases for $l$.
    If $l = 1$, then
    \begin{align*}
        \abs{ \frac{ \alpha^{(l)} - \alpha^{(k)} }{ \alpha^{(l)} - \alpha^{(j)} } } &\leq \abs{ \frac{ \alpha^{(1)} - \alpha^{(2)} }{ \alpha^{(1)} - \alpha^{(3)} } }
        =  \abs{\frac{F_n-L_n + L(10\alpha^{-2n})  }{ F_n + L(9 \alpha^{-2n}) }} < 4.
    \end{align*}
    Analogously, if $l = 2$, then
    \begin{align*}
        \abs{ \frac{ \alpha^{(l)} - \alpha^{(k)} }{ \alpha^{(l)} - \alpha^{(j)} } } &\leq \frac{ L_n + L(7 \alpha^{-2n}) }{ L_n-F_n + L(10 \alpha^{-2n}) } < 4
    \end{align*}
    and if $l = 3$, then
    \begin{align*}
        \abs{ \frac{ \alpha^{(l)} - \alpha^{(k)} }{ \alpha^{(l)} - \alpha^{(j)} } } &\leq \frac{L_n+L(7\alpha^{-2n})}{F_n+L(9\alpha^{-2n})} <4.
    \end{align*}
    We thus have
    \[
        \abs{ \frac{ \alpha^{(l)} - \alpha^{(k)} }{ \alpha^{(l)} - \alpha^{(j)} } } \leq 4
    \]
    in all cases which yields, combined with Inequality \eqref{eq:gamma-ub1},
    \[
        \abs{\gamma} = \abs{ \frac{\beta_j}{\beta_k} } \abs{ \frac{ \alpha^{(l)} - \alpha^{(k)} }{ \alpha^{(l)} - \alpha^{(j)} } } \leq  8\alpha^8\abs{y}^{-3} \cdot \alpha^{-3n}.
    \]
\end{proof}

For $i \in \set{1,2,3}$ let $K_i = \mathbb{Q}(\alpha^{(i)})$. As per Equation \eqref{eq:root-coef}, both
\[
    \epsilon_i := \alpha^{(i)} \; \text{ and } \; \delta_i := \alpha^{(i)} - F_n
\]
are units in $\mathbb{Z}[\alpha^{(i)}]$. By a result of Thomas \cite[Theorem 3.9]{thomas_unit}, $\epsilon_i$ and $\delta_i$ even form a system of fundamental units of $\mathbb{Z}[\alpha^{(i)}]$. The regulator can then be bounded as follows.
\begin{lemma}\label{lem:reg-ub}
    Let
    \[
        G = G_i = \langle -1, \varepsilon_i, \delta_i \rangle = \br{\mathbb{Z}[\alpha^{(i)}]}^\times
    \]
    and $R_G = R_{G_i}$ denote the regulator of $G$, then we have
    \[
        2\br{\log\alpha}^2 n^2 \leq R_G \leq 2n^2.
    \]
\end{lemma}
\begin{proof}
    Since $[K_i:\mathbb{Q}] = 3$ and $K_i$ is a totally real number field, there are three embeddings $\sigma_1,\sigma_2, \sigma_3$ of $K_i$ into $\mathbb{R}$, with $\sigma_1 = \text{id}$. The regulator $R_G$ of $G$ can, up to sign, be expressed as the determinant of the matrix
    \begin{align*}
        \pm R_G = \det\begin{pmatrix}
            \log\abs{\epsilon_i} & \log\abs{\delta_i} \\
            \log\abs{\sigma_2(\epsilon_i)} & \log\abs{\sigma_2(\delta_i)}
        \end{pmatrix}.
    \end{align*}
    
    Since the $\sigma_i$ are $\mathbb{Q}$-automorphisms, we have $\sigma_i(t) = t$ for all rational $t$. Therefore the expression simplifies to
    \begin{align*}
        \pm R_G = \det\begin{pmatrix}
            \log\abs{\alpha^{(i)}} & \log\abs{\alpha^{(i)} - F_n} \\
            \log\abs{\sigma_2(\alpha^{(i)})} & \log\abs{\sigma_2(\alpha^{(i)})-F_n}
        \end{pmatrix}.
    \end{align*}
    Let us denote $\sigma_2(\alpha^{(i)}) = \alpha^{(i')}$. 
    By Lemma \ref{lem:root-log-aprox}, the determinant becomes, up to sign,
    \[
        \pm R_G = 3(\log\alpha)^2 n^2 + L\br{4n}.
    \]
    for each case for $(i,i')$. This gives
    \[
        2\br{\log \alpha}^2 n^2 \leq R_G \leq 2n^2
    \]
    for $n \geq 5$, which proves the asserted bounds for $R_G$.
\end{proof}

Since $\epsilon_i$ and $\delta_i$ are fundamental units in $\mathbb{Z}[\alpha^{(i)}]$, the units $\beta_1, \beta_2, \beta_3$ can be expressed as
\begin{equation}\label{eq:beta-factorize}
    \beta_i = \epsilon_i^{b_1} \delta_i^{b_2}
\end{equation}
for some integers $b_1,b_2$ non-dependent on $i$, since the $\beta_i$ are conjugates. We take the absolute value and the logarithm and get
\begin{equation}\label{eq:beta-log}
    \log\abs{\beta_i} = b_1 \log\abs{\epsilon_i} + b_2 \log\abs{\delta_i}.
\end{equation}
This will lead to a bound for $b_1$ and $b_2$.
\begin{lemma}\label{lem:b-bound}
    Let $\set{j,k,l} = \set{1,2,3}$. Then we have
    \[
        \max\set{\abs{b_1},\abs{b_2}} \leq 7n^{-1} \cdot \max\set{\log\abs{\beta_k},\log\abs{\beta_l}} .
    \]
\end{lemma}
\begin{proof}
    We rewrite Equation \eqref{eq:beta-log} for $k,l$ into the system of linear equations
    \begin{equation}\label{eq:beta-LGS}
        \begin{pmatrix}
            \log\abs{\epsilon_k} & \log\abs{\delta_k} \\
            \log\abs{\epsilon_l} & \log\abs{\delta_l}
        \end{pmatrix}
        \begin{pmatrix}
            b_1 \\ b_2
        \end{pmatrix}
        =
        \begin{pmatrix}
            \log\abs{\beta_k} \\ \log\abs{\beta_l}
        \end{pmatrix}.
    \end{equation}
    We call the above $2\times 2$ matrix $M$. Its determinant is, up to sign, equal to $R_G$ of Lemma \ref{lem:reg-ub}, as such
    \begin{equation}\label{eq:detM-lb}
        \abs{\det M} \geq 2\br{\log\alpha}^2 n^2.
    \end{equation}
    We can thus multiply Equation \eqref{eq:beta-LGS} with $M^{-1}$ and take the matrix-infinity norm $\absn\cdot_\infty$. This yields, in combination with the inequality $\absn{Ax}_\infty \leq \absn{A}_\infty \cdot \absn{ x}_\infty$,
    \begin{equation}\label{eq:b-ub}
        \max\set{\abs{b_1},\abs{b_2}} \leq \absn{M^{-1}}_\infty \max\set{\log\abs{\beta_k},\log\abs{\beta_l}}.
    \end{equation}
    The inverse matrix of $M$ is given by
    \begin{align*}
    \frac{1}{\det M}
        \begin{pmatrix}
            \log\abs{\delta_l} & -\log\abs{\delta_k} \\
            - \log\abs{\epsilon_l} & \log\abs{\epsilon_k}
        \end{pmatrix}
    \end{align*}
    and thus its infinity norm by
    \begin{equation}\label{eq:M-norm}
        \absn{M^{-1}}_\infty = \frac{1}{\abs{\det M}} \max\set{ \abs{\log\abs{\delta_l}}+\abs{\log\abs{\delta_k}}, \abs{\log\abs{\epsilon_l}} + \abs{\log\abs{\epsilon_k}} }.
    \end{equation}
    By Lemma \ref{lem:root-log-aprox}, all terms inside the maximum are linear in $n$. Furthermore, the highest coefficient of $n$ is given by 
    \[
        \abs{ \log\abs{\delta_1}} + \abs{\log\abs{\delta_3} } = 3\log\alpha \cdot n + L(6)
    \]
    and $3\log\alpha \cdot  n + L(6) < 3n$ for $n \geq 10$. This yields
    \[
        \max\set{ \abs{\log\abs{\delta_l}}+\abs{\log\abs{\delta_k}}, \abs{\log\abs{\epsilon_l}} + \abs{\log\abs{\epsilon_k}} } \leq 3n.
    \]
    If we combine this with Equation \eqref{eq:M-norm} and Inequality \eqref{eq:detM-lb}, we get
    \begin{equation*}
        \absn{M^{-1}}_\infty \leq \frac{3}{2} \br{\log\alpha}^{-2} n^{-1} < 7n^{-1}.
    \end{equation*}
    Inequality \eqref{eq:b-ub} thus becomes
    \[
        \max\set{\abs{b_1},\abs{b_2}} \leq 7n^{-1} \max\set{\log\abs{\beta_k},\log\abs{\beta_l}}.
    \]
\end{proof}

\begin{lemma}\label{lem:log-beta_kl-ub}
    For $\set{j,k,l} = \set{1,2,3}$ we have
    \[
        \max\set{\log\abs{\beta_k},\log\abs{\beta_l}} \leq \log\abs{y} + n\log\alpha + 1.
    \]
\end{lemma}
\begin{proof}
    Let 
    \[
        (X-F_n)(X-L_n)X - 1 = (X-p_1)(X-p_2)(X-p_3) - 1,
    \]
    that is $p_1 = F_n, p_2 = L_n, p_3 = 0$. For $i \neq j$, we then have
    \[
        \abs{\frac{\beta_i}{y}} = \abs{\frac{x}{y}-\alpha^{(i)}} \leq \abs{\frac{x}{y}-\alpha^{(j)}} + \abs{\alpha^{(j)}-p_j} + \abs{p_j - \alpha^{(i)}}
    \]
    due to the triangle inequality. Factorising the right-hand side leads to
    \begin{equation}\label{eq:betai-ai-pj}
        \abs{\frac{\beta_i}{y}} \leq \abs{\alpha^{(i)} - p_j} \br{ 1 + \frac{\abs{\alpha^{(j)} - p_j}}{\abs{\alpha^{(i)}-p_j}} + \frac{\abs{\beta_j/y}}{\abs{\alpha^{(i)}-p_j}} }.
    \end{equation}
    We estimate the terms inside the bracket. Using Lemma \ref{lem:root-aprox}, we get
    \[
        \abs{\alpha^{(i)}-p_j} \geq \br{ F_n + L(9\alpha^{-2n} }
    \]
    and thus, by Inequality \eqref{eq:fib-bnd},
    \begin{equation}\label{eq:ai-pj-lb}
        \abs{\alpha^{(i)}-p_j} \geq \alpha^{n-2} - 9\alpha^{-2n} \geq \alpha^{n-3} = \alpha^{-3} \alpha^n
    \end{equation}
    on the one hand, and
    \begin{equation}\label{eq:ai-pj-ub}
        \abs{\alpha^{(i)}-p_j} 
        \leq \br{L_n+L\br{4\alpha^{-2n}} - 0} \leq \alpha^n + \alpha^{n-2} + 4 \alpha^{-2n}
    \end{equation}
    by Inequality \eqref{eq:fib-bnd} on the other. Moreover, the root $\alpha^{(j)}$ is close to the coefficient $p_j$, by Lemma \ref{lem:root-aprox} as close as
    \begin{equation}\label{eq:aj-pj-ub}
        \abs{\alpha^{(j)} - p_j} \leq \abs{ F_n + L\br{6\alpha^{-2n}} - F_n } \leq 6 \alpha^{-2n}.
    \end{equation}
    Since we assume that $\abs{y}\geq 2$, we have $\abs{\beta_j} \leq \frac{1}{2}\alpha^5\alpha^{-2n}$  by Lemma \ref{lem:beta_j-ub}. Combining this with Inequalities \eqref{eq:ai-pj-lb} and \eqref{eq:aj-pj-ub} yields
    \begin{align}\label{eq:as-ps-ub}
        \br{ 1 + \frac{\abs{\alpha^{(j)} - p_j}}{\abs{\alpha^{(i)}-p_j}} + \frac{\abs{\beta_j/y}}{\abs{\alpha^{(i)}-p_j}} } &\leq \br{ 1 + 6\alpha^3 \alpha^{-3n} + \frac{1}{4}\alpha^8\alpha^{-3n} } \nonumber \\
        &\leq \br{1 + 38 \alpha^{-3n}}.
    \end{align}
    We combine Inequalities \eqref{eq:betai-ai-pj} and \eqref{eq:as-ps-ub}, multiply with $\abs{y}$ and arrive at
    \[
        \abs{\beta_i} \leq \abs{\alpha^{(i)}-p_j} \cdot |y| \cdot \br{1+38\alpha^{-3n}}.
    \]
    Using Inequality \eqref{eq:ai-pj-ub} and taking logarithms yields
    \[
        \log\abs{\beta_i} \leq \log\br{\alpha^n + \alpha^{n-2} + 4\alpha^{-2n} } + \log\abs{y} + \log\br{1+38\alpha^{-3n}}.
    \]
    We then make full use of Lemma \ref{lem:log-trick} and arrive at
    \[
        \log\abs{\beta_i} \leq \log\br{\alpha^n} + 2\br{\alpha^{-2}+4\alpha^{-3n}} + \log\abs{y} + 76\alpha^{-3n}
    \]
    and thus
    \[
        \log\abs{\beta_i} \leq \log\abs{y} + n \log\alpha + 1.
    \]
\end{proof}

\section{Linear Forms in logarithms and an effective bound for n}

We go back to Equation \eqref{eq:siegl},
\[
    \frac{\beta_l}{\beta_k} \frac{ \alpha^{(j)} - \alpha^{(k)} }{ \alpha^{(j)} - \alpha^{(l)} } - 1 = \gamma,    
\]
and use Equation \eqref{eq:beta-factorize} to rewrite the expression as
\[
    \br{\frac{\epsilon_l}{\epsilon_k}}^{b_1} \br{\frac{\delta_l}{\delta_k}}^{b_2} \frac{ \alpha^{(j)} - \alpha^{(k)} }{ \alpha^{(j)} - \alpha^{(l)} } = 1+\gamma.
\]
We then define the linear form in logarithms $\Uplambda$ by
\[
    \Uplambda = \log\abs{1+\gamma} = b_1 \log\abs{\frac{\epsilon_l}{\epsilon_k}} + b_2 \log\abs{\frac{\delta_l}{\delta_k}} + \log\abs{\frac{ \alpha^{(j)} - \alpha^{(k)} }{ \alpha^{(j)} - \alpha^{(l)} }}.
\]
Due to Lemma \ref{lem:log-ub} and \ref{lem:gamma-ub} and $\abs{y}\geq 2$, we have
\begin{equation}\label{eq:lambda-ub}
    \abs{\Uplambda} \leq 2 \abs{\gamma} \leq 16\alpha^8\abs{y}^{-3} \cdot \alpha^{-3n} \leq 2\alpha^8 \alpha^{-3n}.
\end{equation}

We want to immediately apply Theorem \ref{thm:baker-lb} to bound $\abs{\Uplambda}$ from below. The problem with this, however, lies in the expressions $\epsilon_l/\epsilon_k$ and $\delta_l/\delta_k$ having huge heights that would render the bound useless for our purposes. Therefore, we first rewrite the linear form to get the (exponential) dependency on $n$ out of the logarithms.

To that end we need to differentiate the cases for $j$. Let $j=1$ and choose $k=3, l=2$.

By Lemma $\ref{lem:root-aprox}$ we get
\begin{align*}
    \frac{\epsilon_2}{\epsilon_3} = \frac{\alpha^{(2)}}{\alpha^{(3)}} &= \frac{\alpha^n \pm \alpha^{-n} + L\br{4 \alpha^{-2n}} }{\sqrt{5} \alpha^{-2n} + L\br{\alpha^{-4n}} } =  \frac{\alpha^n \br{1 \pm \alpha^{-2n} + L\br{4 \alpha^{-3n}} }}{\sqrt{5} \alpha^{-2n} \br{ 1 + L\br{\sqrt{5}^{-1}\alpha^{-2n}}} }\\
    &= \frac{1}{\sqrt{5}} \alpha^{3n} \br{ 1 + L\br{\frac{3}{2} \alpha^{-2n}} },
\end{align*}
and thus, combined with Lemma \ref{lem:log-trick},
\begin{equation}\label{eq:log-epsfrac-aprox}
    \log\abs{\frac{\epsilon_2}{\epsilon_3}} = 3n \log \alpha - \log \sqrt{5} + L(3 \alpha^{-2n}).
\end{equation}
Similarily,
\begin{align*}
    \frac{\delta_2}{\delta_3} &= \frac{ \br{ 1- \frac{1}{\sqrt{5}} } \alpha^n \pm \br{ 1+ \frac{1}{\sqrt{5}} } \alpha^{-n} + L(10 \alpha^{-2n}) }{ - \frac{1}{\sqrt{5}} \alpha^n \pm \frac{1}{\sqrt{5}} \alpha^{-n} + L(3 \alpha^{-2n}) } \\
    &= \br{1-\sqrt{5}} \br{ 1 + L\br{\frac{5}{2} \alpha^{-2n}} }
\end{align*}
and thus
\begin{equation}\label{eq:log-deltfrac-aprox}
    \log\abs{\frac{\delta_2}{\delta_3}} = \log\br{\sqrt{5}-1} + L(5 \alpha^{-2n}).
\end{equation}
Lastly,
\begin{align*}
    \frac{ \alpha^{(1)}-\alpha^{(3)} }{ \alpha^{(1)}-\alpha^{(2)} } &= \frac{ \frac{1}{\sqrt{5}} \alpha^n \pm \frac{1}{\sqrt{5}} \alpha^{-n} + L(9 \alpha^{-2n}) }{ \br{\frac{1}{\sqrt{5}}-1} \alpha^n \pm \br{\frac{1}{\sqrt{5}}+1} \alpha^{-n} + L(10\alpha^{-2n}) } \\
    &= \frac{1}{1-\sqrt{5}} \br{ 1 + L\br{\frac{3}{2} \alpha^{-2n}} }
\end{align*}
and thus
\begin{equation}\label{eq:log-alphfrac-aprox}
    \log\abs{\frac{ \alpha^{(1)}-\alpha^{(3)} }{ \alpha^{(1)}-\alpha^{(2)} }} = -\log(\sqrt{5}-1) + L(3\alpha^{-2n}).
\end{equation}
If we combine Equations \eqref{eq:log-epsfrac-aprox}-\eqref{eq:log-alphfrac-aprox} and shift the $L$-terms into the upper bound, then the original linear form
\[
    \abs{\Uplambda} = \abs{b_1 \log\abs{\frac{\epsilon_l}{\epsilon_k}} + b_2 \log\abs{\frac{\delta_l}{\delta_k}} + \log\abs{\frac{ \alpha^{(j)} - \alpha^{(k)} }{ \alpha^{(j)} - \alpha^{(l)} }}} \leq 2\alpha^8 \alpha^{-3n}
\]
transforms into a linear form
\[
    \xi = x_1 \log\alpha + x_2 \log\sqrt{5} + x_3 \log\br{\sqrt{5}-1} 
\]
with
\begin{equation}\label{eq:x_i-case1}
    x_1 = 3nb_1, \; x_2 = -b_1, \; x_3 = b_2 - 1.
\end{equation}

Analogously, the linear form $\xi$ has the same logarithms in the case $j=2$, where we choose $(k,l) = (1,3)$, and in the case $j = 3$, where we choose $(k,l) = (1,2)$. The coefficients are
\begin{equation}\label{eq:x_i-case2}
    x_1 = -3nb_1 + 3nb_2 , \; x_2 = 2b_1-3b_2-1, \; x_3 = b_2 + 1
\end{equation}
and
\begin{equation}\label{eq:x_i-case3}
    x_1 = 3nb_2, \; x_2 = b_1-3b_2-1, \; x_3 = 2b_2
\end{equation}    
respectively.

In all three cases, the respective L-terms can be bounded from above by the term ${(3\abs{b_1}+5\abs{b_2}+5) \alpha^{-2n}}$, thus
\begin{equation}\label{eq:xi-ub}
    \abs{\xi} \leq 2\alpha^8 \alpha^{-3n} + \br{3\abs{b_1}+5\abs{b_2}+5} \alpha^{-2n}.
\end{equation}

\begin{lemma}
    The linear form $\xi$ fulfils
    \[
        \xi \neq 0
    \]
    in all cases.
\end{lemma}
\begin{proof}
    Since $\log\alpha, \log\sqrt{5}, \log\br{\sqrt{5}-1}$ are $\mathbb{Q}$-linearly independent, $\xi = 0$ would imply ${x_1=x_2=x_3 = 0}$.
    
    In the case $j=1$, this implies $b_1 = 0, b_2 = 1$. By Equation \eqref{eq:beta-factorize}, this means
    \[
        x-\alpha^{(i)}y = \alpha^{(i)} - F_n
    \]
    and thus $(x,y) = -(F_n,1)$, which we ruled out.
    
    In the case $j=2$, the expression $x_1 = 0$ implies $b_1 = b_2$. It then follows from $x_2 = 0$, that $b_1 = b_2 = -1$.
    Thus, Equation \eqref{eq:beta-factorize} implies
    \[
        x - \alpha^{(i)}y = \br{ \alpha^{(i)} }^{-1} \br{ \alpha^{(i)} - F_n }^{-1}.
    \]
    This implies
    \[
        \br{ -\alpha^{(i)} y + x }\br{ \alpha^{(i)} - F_n } \br{\alpha^{(i)}} - 1 = 0
    \]
    for each $\alpha^{(i)}$, which means that the $\alpha^{(i)}$ are roots of the polynomial
    \[
        \br{-Xy+x}\br{X-F_n}X - 1.
    \]
    The minimal polynomial of the $\alpha^{(i)}$ is given by Equation \eqref{eq:defining-polynomial}, however, and it follows that $(x,y) = -(L_n,1)$, which again we ruled out.
    
    In the case $j=3$, it follows from $x_1 = 0$ that $b_2 = 0$ and from $x_2 = 0$ that $b_1 = 1$. Equation \eqref{eq:beta-factorize} then implies
    \[
        x-\alpha^{(i)}y = \alpha^{(i)}
    \]
    and thus $(x,y) = -(0,1)$, which we also ruled out.
\end{proof}

Before we can apply the lower bound from Theorem $\ref{thm:baker-lb}$, we have to bound the coefficients $x_i$ of $\xi$ from above: 

\begin{lemma}\label{lem:x_i-ub}
    We have
    \[
        \max\set{\abs{x_1},\abs{x_2},\abs{x_3}} \leq 42 \log\abs{y} + 42\log\alpha \cdot n + 43.
    \]
\end{lemma}
\begin{proof}
    By looking at Equations \eqref{eq:x_i-case1}, \eqref{eq:x_i-case2} and \eqref{eq:x_i-case3} we have
    \[
        \max\set{\abs{x_1},\abs{x_2},\abs{x_3}} \leq 6n \max\set{\abs{b_1},\abs{b_2}} + 1
    \]
    for $n \geq 1$ in all cases. Additionally, Lemma \ref{lem:b-bound} gives us
    \[
        \max\set{\abs{x_1},\abs{x_2},\abs{x_3}} \leq 42 \cdot \max\set{\log\abs{\beta_k},\log\abs{\beta_l}} + 1.
    \]
    By Lemma \ref{lem:log-beta_kl-ub}, this then yields
    \[
        \max\set{\abs{x_1},\abs{x_2},\abs{x_3}} \leq 42 \cdot \br{ \log\abs{y} + n\log\alpha + 1 } + 1
    \]
    and thus
    \[
        \max\set{\abs{x_1},\abs{x_2},\abs{x_3}} \leq 42 \log\abs{y} + 42\log\alpha \cdot n + 43.
    \]
\end{proof}

We now use Baker's and Wüstholz's lower bound (Theorem \ref{thm:baker-lb}) for the linear form $\xi$ to derive an exponential lower bound for $\log\abs{y}$.

\begin{lemma}\label{lem:logy-lb}
    One has
    \[
        \log\abs{y} \geq \exp \br{ \frac{2 \log \alpha}{1+C} \cdot n - \log n - 5}
    \]
    with $C = 17496\cdot 64^5 \cdot \log\br{12} \cdot \log\alpha \cdot \log\sqrt{5} \cdot \log 2 \approx 1.253\cdot10^{13}$.
\end{lemma}
\begin{proof}
    By Theorem \ref{thm:baker-lb}, we have
    \begin{align*}
        \log\abs{\xi} &\geq - 18(3+1)!3^{3+1}(32\cdot 2)^{3+2}\log(2\cdot 3 \cdot 2) h_1 \cdots h_t \log B \\
        &= -34992 \cdot 64^5 \cdot \log(12) \cdot h_1 h_2 h_3 \cdot \log \max\set{\abs{x_1},\abs{x_2},\abs{x_3}},
    \end{align*}
    since $\xi$ is a linear form in the $t=3$ logarithms $\log\alpha, \log\sqrt{5}, \log\br{\sqrt{5}-1}$ and ${D = [\mathbb{Q}(\alpha, \sqrt{5}, \sqrt{5}-1):\mathbb{Q}]} = [\mathbb{Q}(\sqrt{5}):\mathbb{Q}] = 2$. 
    
    The minimal polynomials of $\alpha, \sqrt{5}$ and $\sqrt{5}-1$ are given by
    \begin{align*}
        X^2-X-1 &= (X-\alpha)(X-\beta) \\
        X^2-5 &= (X-\sqrt{5})(X+\sqrt{5}) \\
        X^2+2X-4 &= (X-(\sqrt{5}-1))(X+(\sqrt{5}+1))
    \end{align*}
    respectively. Thus, the logarithmic heights are $\frac{1}{2}\log\alpha, \log\sqrt{5}$ and $\log 2$ respectively.
    Together with Lemma \ref{lem:x_i-ub}, this yields
    \[
        \log\abs{\xi} \geq -C \cdot \log\br{ 42 \log\abs{y} + 42\log\alpha \cdot n + 43 }
    \]
    with $C = 17496\cdot 64^5 \cdot \log\br{12} \cdot \log\alpha \cdot \log\sqrt{5} \cdot \log 2 \approx 1.253\cdot10^{13}$. We estimate the last term using $\log\abs{y} \geq \log 2$ and $n\geq 10$:
    \begin{align*}
        &\log\br{ 42 \log\abs{y} + 42\log\alpha \cdot n + 43 } \\
        = &\log\log\abs{y} + \log\br{ 42 + \frac{42}{\log\abs{y}}\log\alpha \cdot n + \frac{43}{\log\abs{y}} } \\
       \leq & \log\log\abs{y} + \log\br{30n+103} \leq \log n + \log\log\abs{y} + 4.
    \end{align*}
    Thus, the lower bound for $\log\abs{\xi}$ becomes
    \begin{equation}\label{eq:xi-lb}
        \log\abs{\xi} \geq -C \br{ \log n + \log\log\abs{y} + 4 }.
    \end{equation}

    On the other hand, recall that Inequality \eqref{eq:xi-ub} gives us an upper bound for $\abs{\xi}$, namely
    \[
        \abs{\xi} \leq 2\alpha^8 \alpha^{-3n} + (3\abs{b_1}+5\abs{b_2}+5) \alpha^{-2n}.
    \]
    Plugging in the bound from Lemma \ref{lem:b-bound} gives us
    \[
        \abs{\xi} \leq 2\alpha^8 \alpha^{-3n} + \br{ 56 n^{-1}  \max\set{\log\abs{\beta_k},\log\abs{\beta_l}} + 5 } \alpha^{-2n}.
    \]
    By Lemma \ref{lem:log-beta_kl-ub} and with $n\geq 10$, this then yields
    \begin{align}\label{eq:xi-ub2}
        \abs{\xi} &\leq \alpha^{-2n} \br{ 2\alpha^8\alpha^{-n} + 56n^{-1}\log\abs{y} + 56\log\alpha + 56n^{-1} + 5 } \nonumber \\
        &\leq \alpha^{-2n}\br{ 6 \log\abs{y} + 39 }.
    \end{align}
    We take the logarithm of Inequality \eqref{eq:xi-ub2} and compare it with Inequality \eqref{eq:xi-lb}, which yields
    \begin{align}
        -C \br{\log n + \log\log\abs{y} + 4} &\leq -2n\log\alpha + \log\br{ 6 \log\abs{y} + 39 } \nonumber \\
        &\leq -2n \log\alpha + \log\log\abs{y} + 5. \nonumber
    \end{align}
    We rewrite this into
    \[
        \log\log\abs{y} \br{1+C} \geq 2\log\alpha \cdot n - C \log n - \br{5+4C}
    \]
    and thus get
    \[
        \log\abs{y} \geq \exp \br{ \frac{2 \log \alpha}{1+C} \cdot n - \log n - 5}.
    \]
\end{proof}

Next, we use Bugeaud's and Győry's upper bound (Theorem \ref{thm:bugy-ub}) for $\log\abs{y}$.
\begin{lemma}\label{lem:logy-ub}
    We have
    \[
        \log\abs{y} \leq 3^{94} \cdot 2n^2 \cdot \log\br{2n^2} \cdot \br{ 2n^2 + (2n-1) \log\alpha + 1 }.
    \]
    \begin{proof}
        Let $K = \mathbb{Q}(\alpha^{(1)})$. By Lemma \ref{lem:reg-ub} we have
        \[
            R_K \leq 2n^2
        \]
        for the regulator $R_K$. The unit-rank is $r = 2$, an upper bound for the coefficients of the Thue Equation \eqref{eq:thue} is given by $H = \alpha^{2n-1} \geq L_n\cdot F_n = F_{2n}$ by Equation \eqref{eq:fib-bnd} and, finally, $B = \max\{1,e\}=e$.
        
        Theorem $\ref{thm:bugy-ub}$ then yields
        \begin{align*}
            \log\abs{y} &\leq 3^{2+27} \cdot 3^{14+19} 3^{6+12+14} \cdot \br{2n^2} \cdot \log\br{2n^2} \cdot\br{ 2n^2 + \log\br{ \alpha^{2n-1}  e } } \\
            &= 3^{94} \cdot 2n^2 \cdot \log\br{2n^2} \cdot \br{ 2n^2 + (2n-1) \log\alpha + 1 }.
        \end{align*}
    \end{proof}
\end{lemma}

If we compare the lower bound from Lemma \ref{lem:logy-lb} and the upper bound from Lemma \ref{lem:logy-ub}, we get the following absolute bound for $n$.
\begin{lemma}\label{lem:n-ub1}
We have
\[
    n \leq 1.144 \cdot 10^{15}.
\]
\end{lemma}

\section{Reducing the bound for n}

We want to reduce the bound for $n$ using the LLL-algorithm as described in \cite[Lemma VI.1]{smart_book}.

To that end we take a look at the linear form $\xi$ and rewrite its upper bound from Inequality \eqref{eq:xi-ub2}, namely
\begin{align*}
    \abs{\xi} &= \abs{ x_1 \log\alpha + x_2 \log \sqrt{5} + x_3 \log\br{\sqrt{5}-1} } \\
    &\leq \alpha^{-2n} \br{ 6\log\abs{y} + 39 } = c_2 e^{-c_3 \cdot n},
\end{align*}
where $c_3 := 2\log\alpha$ and $c_2$ is given by Lemma \ref{lem:logy-ub} to be
\[
    c_2(n) := 6\cdot 3^{94}  \cdot 2n^2 \cdot \log\br{2n^2} \cdot \br{2n^2+(2n-1)\log\alpha + 1} + 39,
\]
which gives $c_2 = c_2\br{1.144\cdot10^{15} } \approx 2.036 \cdot 10^{108}$ when using the bound for $n$ found in Lemma \ref{lem:n-ub1}. 

We can bound $x_i$ for each $i \in \set{1,2,3}$ from above by Lemma \ref{lem:x_i-ub}, together with the bound from Lemma \ref{lem:logy-ub} for $\log\abs{y}$ and Lemma \ref{lem:n-ub1} for $n$. This gives
\[
   \abs{x_i} \leq c_{x_i}(n) := 42 \cdot 3^{94}  \cdot 2n^2 \cdot \log\br{2n^2} \cdot \br{2n^2+(2n-1)\log\alpha + 1} + 42 \log\alpha \cdot n + 43
\]
and $c_{x_i} = c_{x_i}\br{ 1.144\cdot 10^{15}} \approx 1.425 \cdot 10^{109} $. We set $c = \br{\max_{i \in \set{1,2,3}} c_{x_i}}^3$ and then form the matrix
\begin{align*}
    \begin{pmatrix}
        1 & 0 & 0 \\
        0 & 1 & 0 \\
        [c\log\alpha] & [c \log\sqrt{5}] & [c\log\br{\sqrt{5}-1}]
    \end{pmatrix},
\end{align*}
where $[\cdot]$ denotes the nearest integer. We use the LLL-algorithm in Sage \cite{sage} on this matrix and set $c_4^2$  to be the maximum column norm of the resulting matrix.

If we set $S$ to be $S = \sum_{i=1}^2 c_{x_i}^2$ and $T = \br{ 1+\sum_{i=1}^3c_{x_i} }/2$, then we have ${c_4^2 \geq T^2+S}$. Lemma VI.1 of \cite{smart_book} then yields
\[
    n \leq \frac{1}{c_3} \br{ \log(c\cdot c_2) - \log \br{\sqrt{c_4^2-S}-T}}
\]
and thus $n \leq 650$. We iterate the above process to get a further reduced bound each time:

\begin{align*}
    n &\leq 1.144\cdot 10^{15} \longrightarrow n \leq 650 \longrightarrow n \leq 353 \longrightarrow
    n \leq 346,
\end{align*}
after which we get no further reduction.

The bound $n \leq 346$ is still too large to simply tackle the remaining cases with brute-force. We let Sage with its interface for PARI/GP \cite{pari} solve some of the equations and find the non-trivial solutions for $n\in \set{1,3}$ as stated in Theorem $\ref{thm:main}$. However, looping over all possible values for $n$ quickly proves impossible -- we aborted the loop on iteration $n=49$ after being stuck on it for over $8$ hours.

We instead tackle the last remaining cases $49 \leq n \leq 346$ as follows. We look at the original linear form $\Uplambda$ and its upper bound
\begin{equation*}
    \abs{\Uplambda} = \abs{b_1 \log\abs{\frac{\epsilon_l}{\epsilon_k}} + b_2 \log\abs{\frac{\delta_l}{\delta_k}} + \log\abs{\frac{ \alpha^{(j)} - \alpha^{(k)} }{ \alpha^{(j)} - \alpha^{(l)} }}} \leq 2\alpha^8 \alpha^{-3n} = c_2 e^{-c_3 \cdot n}
\end{equation*}
with $c_2 = 2\alpha^8$ and $c_3 = 3\log\alpha$.  The bounds for the coefficients of $\Uplambda$ are now
\begin{align*}
    c_{b_1},c_{b_2} := 7n^{-1} \br{ 3^{94}2n^2\log\br{2n^2}\br{2n^2+(2n-1)\log\alpha + 1} + n\log\alpha + 1 }
\end{align*}
according to Lemma \ref{lem:b-bound}, \ref{lem:log-beta_kl-ub} and \ref{lem:logy-ub}. The $\alpha^{(i)}$, and thus the logarithms in $\Uplambda$, are estimated numerically using a bit-precision of $3n+30$ and Newton descent with an error tolerance of $10^{-200}$.

If we iterate the LLL-reduction method with the above configuration, we can further reduce the bound to $n \leq 132$. The one last problem that hinders us is that $c_3$ is a bit too small. If we had $c_3 = 9\log\alpha$ instead, the LLL-reduction would give us a small enough bound for $n$ and we would be done. To that end, we ascertained:

\begin{lemma}
    For $49 \leq n \leq 132$, one has
    \begin{equation*}
        \abs{y} \geq \alpha^{2n}.
    \end{equation*}
\end{lemma}
\begin{proof}
    It follows immediately from Lemma \ref{lem:beta_j-ub} that $\frac{x}{y}$ is a convergent to $\alpha^{(j)}$. We checked computationally that for no $49\leq n \leq 132$ and no convergent $\frac{x}{y}$ of either $\alpha^{(1)}, \alpha^{(2)}, \alpha^{(3)}$ -- numerically estimated -- with denominator $y < \alpha^{2n}$, the pair $(x,y)$ provides a solution to the associated Thue equation.
\end{proof}
Using now the fact that $\abs{y} \geq \alpha^{2n}$ instead of just $\abs{y}\geq 2$, the upper bound for $\Uplambda$ in Inequality \eqref{eq:lambda-ub} becomes
\begin{equation*}
    \abs{\Uplambda} \leq 16\alpha^8 \abs{y}^{-3} \alpha^{-3n} \leq 16\alpha^8 \alpha^{-9n}.
\end{equation*}
Using $c_2 = 16\alpha^8$ and $c_3 = 9\log\alpha$ in the LLL-algorithm now gives $n \leq 44$. We have no further cases to consider and thus conclude the proof of Theorem \ref{thm:main}.

\section*{Acknowledgement}
The second and third author were supported by the Austrian Science Fund (FWF) under the project I4406.

\end{document}